\documentclass[12pt,reqno]{amsart}
\usepackage{amssymb}
\usepackage{amsmath, mathtools}

\usepackage{amsthm}

\usepackage{amscd}

\newcommand{\RNum}[1]{\uppercase\expandafter{\romannumeral #1\relax}}

\usepackage[T2A]{fontenc}
\usepackage[utf8]{inputenc}
\usepackage[english]{babel}

\input{int.def} 

\usepackage[sort]{cite}
\usepackage{tikz-cd}
\usetikzlibrary{cd}
\usepackage{dirtytalk}
\usepackage[linktoc=page, colorlinks, linkcolor=blue, citecolor=blue]{hyperref}

\usepackage{xcolor}
\usepackage{centernot}

\usepackage{enumitem}

\usepackage{pgfplots}
\usepackage{multicol}

\pgfplotsset{compat=1.17}

\makeatletter
\renewenvironment{proof}[1][\proofname]{%
  \par\vspace{\topsep}%
  \normalfont\topsep6\p@\@plus6\p@\relax
  \trivlist
  \item[\hskip\labelsep\itshape #1\@addpunct{.}]\ignorespaces
}{%
  \endtrivlist
}
\makeatother

\numberwithin{equation}{section}

\DeclarePairedDelimiter \abs{\lvert}{\rvert} 
\DeclarePairedDelimiter \norm{\lVert}{\rVert}
\DeclarePairedDelimiterX \ip[2]{\langle}{\rangle}{#1,#2}
\DeclarePairedDelimiterXPP \Prob[1]{\mathbb{P}}\{\}{}{\newcommand\given{\nonscript\:\delimsize\vert\nonscript\:\mathopen{}} #1} 
\DeclarePairedDelimiterXPP \Probevent[1]{\mathbb{P}}(){}{\newcommand\given{\nonscript\:\delimsize\vert\nonscript\:\mathopen{}}#1} 

\DeclareMathOperator{\disc}{disc}
\DeclareMathOperator{\E}{\mathbb{E}}

\newcommand{\one}{\mathbf{1}}

\def \R {\mathbb{R}}
\def \e {\varepsilon}

\usepackage[mathcal]{euscript}

\usepackage{titlesec}
\titleformat{\section}[runin]{\bfseries}{\thesection.}{3pt}{}[.]

\usepackage{geometry}
\newgeometry{vmargin={25mm}, hmargin={22mm,22mm}, footskip=10mm}   

\begin{document}

\title[
Thinning to improve two-sample discrepancy
]{Thinning to improve two-sample discrepancy}

\author{Gleb Smirnov}
\address{
Mathematical Sciences Institute, 
Australian National University, Canberra, Australia}
\email{gleb.smirnov@anu.edu.au}

\author{Roman Vershynin}
\address{Department of Mathematics, University of California, Irvine, US}
\curraddr{}
\email{rvershyn@uci.edu}

\thanks{R.V.~is partially supported by NSF Grant DMS 1954233 and NSF+Simons Research Collaborations on the Mathematical and Scientific Foundations of Deep Learning}


\begin{abstract}
    The discrepancy between two independent samples \(X_1,\dots,X_n\) and \(Y_1,\dots,Y_n\) drawn from the same distribution on $\R^d$ typically has order \(O(\sqrt{n})\) even in one dimension. We give a simple online algorithm that reduces the discrepancy to \(O(\log^{2d} n)\) by discarding a small fraction of the points.
\end{abstract}

\maketitle

\setcounter{section}{0}

Random sampling inevitably brings errors. Classical work in statistics and probability has led to a thorough understanding of sampling errors. Much less is known about how to reduce them with minimal intervention. In this paper, we address the following version of the problem -- how to align two independent random samples with each other by discarding a few points?

\section{Main results} \label{intro}
Let \( X_1, \dots, X_n \) and \( Y_1, \dots, Y_n \) be two independent samples, each consisting of \(n\) independent observations drawn from the same Borel probability distribution on $\R^d$. For a Borel set $B \subset \R^d$, we count how many points from each sample fall into $B$, and define its discrepancy as:
\[
\operatorname{disc}(B) 
\coloneqq \abs[\big]{\#\{i:\;  X_i \in B \} - \#\{i:\; Y_i \in B \}}.
\]
Taking the supremum over all axis-aligned, half-infinite\footnote{Our main results remain valid if the supremum is taken over all finite boxes \(B = [a_1,b_1] \times \cdots \times [a_d,b_d]\).} boxes \(B = (-\infty,b_1] \times \cdots \times (-\infty,b_d]\), define:
\begin{equation}\label{Dnn}
    D_{n,n} \coloneqq \sup_B \operatorname{disc}(B).
\end{equation}
In dimension $d = 1$, the normalized discrepancy \(n^{-1} D_{n,n}\) recovers the Kolmogorov-Smirnov distance. Classical VC theory \cite[Theorem 8.3.23]{vershynin2018high} gives:
$$
\E D_{n,n} = O_d \big(\sqrt{n}\big).
$$
This order is optimal for any nontrivial distribution,  because we have $\E \operatorname{disc}(B) \asymp \sqrt{n}$ for any set $B$ whose measure is bounded away from $0$ and $1$. However, the discrepancy can be drastically reduced by discarding a small fraction of the samples. This is our main result.

\begin{theorem}[Two-sample discrepancy] \label{main}
    Fix \(T\) so that \(1 \leq T \leq \sqrt{n}\). Let \( X_1, \dots, X_n \) and \( Y_1, \dots, Y_n \) be i.i.d. samples from the same Borel probability distribution on $\R^d$. There is a randomized online algorithm that discards, on average, at most $Cn/T$ of the \( X_i \)'s and \( Y_i \)'s, and achieves
    $$
    \E D_{n,n} \le T \log_2^{2d} n,
    $$
    where $C$ is an absolute constant. Expectations are over both samples and the algorithm.    
\end{theorem}
\smallskip%

The algorithm (described in \S\,\ref{proof}) does not need to know the distribution. It is online in the sense that it processes points one by one, with the location of each point revealed only when processed. The decision to keep or discard a point depends only on its value and past decisions, and is final. At each step, the algorithm randomly chooses whether to process the next point from the \(X\)- or \(Y\)-sample by tossing a fair \(\pm 1\) coin: if \(+1\), a point from \(X\); if \(-1\), from \(Y\). The procedure stops when one sample is exhausted; all remaining points from the other are discarded. On average, \(O(\sqrt{n})\) points are left unprocessed, which is absorbed into the \(O(n/T)\) discard budget.
\smallskip%

Let \(x_1, \ldots, x_k \in \R^d\), with \(k < 2n\), denote the sequence of processed points. For each \(i\), define $\e_i = 1$ if $x_i$ is from the $X$-sample and $\e_i = -1$ if $x_i$ is from the $Y$-sample. Then \(\varepsilon_1, \ldots, \varepsilon_k\) are i.i.d. Rademacher random variables. For any Borel set \(B \subset \R^d\), define the {\em sign discrepancy}:
\begin{equation}    \label{eq: sign discrepancy}
    \disc_k(B) 
    = \abs[\Big]{\sum_{i = 1}^{k} \e_i \one_{\{x_i \in B\}}}.
\end{equation}
This reduces Theorem~\ref{main} to the statement 
below.

\begin{proposition}[Sign discrepancy] \label{stream}
    Let \(m\) be an integer with \(m \le 2n\).
    Let $x_1, \dots, x_{m}$ be i.i.d. samples from a Borel probability distribution in $\R^d$, let $\e_1,\ldots,\e_{m}$ be independent Rademacher random variables, and let $T>0$. There is an online algorithm that discards, on average, at most $2n/T$ of the \( x_i \)'s, and achieves 
    $$
    \E \sup_{B, 1 \leq k \leq m} \disc_k(B) 
    \le T \log_2^{2d} n + d \ln(dn) + C_1 d,
    $$
    where the supremum is over all axis-aligned,  half-infinite boxes $B$ and all $k=1,\ldots,m$, and where $C_1$ is an absolute constant. The ``online'' means that, at time $i$, the algorithm decides whether to discard the term $x_i$ based on this and previously retained terms; the decision is never reversed in the future.
\end{proposition}
Specializing to \(m = 2n\) recovers Theorem \ref{main}.

\begin{remark}[Time and memory cost]
    The proof shows that each point can be processed in time \(O(\log_2^d n)\), independent of the parameter $T$, yielding a total time complexity of \(O(n \log_2^d n)\). The algorithm also requires \(O(n \log_2^d n)\) memory.
\end{remark}
\smallskip%

For context, 
recall the notion of 
\emph{star discrepancy}. Let $\mu$ be a Borel probability measure on $\R^d$. 
Let $x_1, \ldots, x_n \in \R^d$. The star discrepancy of the point set $x_1,\ldots,x_n$ with 
respect to $\mu$ is defined as:
\[
D^*(x_1,\ldots,x_n) \coloneqq
\sup_{B} \left| \#\{i:\; x_i \in B\} - n\,\mu(B) \right|,
\]
where the supremum is taken over 
all boxes \(B = (-\infty,b_1] \times \cdots \times (-\infty,b_d]\).
\smallskip%

Constructing point sets with 
low discrepancy (star and two-sample) 
is a classical, well-studied problem; see the monographs \cite{Mat,Chaz,Beck-Chen}. Less is known about how to \emph{improve} the discrepancy of a given random sample by moving or discarding only a few points. Recent progress appears in 
\cite{DFGGR, B-S, CDP, BMMS, Sm-Versh}. 
\smallskip%

Closest to our setting is the work \cite{DFGGR}, which considers the star discrepancy. The authors provide an online algorithm for points sampled from the uniform distribution on \([0,1]^d\); their algorithm discards \(O(n/T)\) points and achieves star discrepancy \(O(T \log^{2d} n)\). Applied separately to the samples \(X\) and \(Y\), and combined via the triangle inequality, this yields a version of Theorem~\ref{main}. The proof in \cite{DFGGR} relies on the orthogonality of Haar wavelets and is thus limited to the uniform measure on the cube. Our approach avoids orthogonality and applies to arbitrary Borel distributions on \(\mathbb{R}^d\). 
\smallskip%

For results on general measures, we refer to \cite{Aist-B-Nik, Aist-Dick}, which construct low-discrepancy sequences without assumptions on the underlying distribution. However, these constructions do not rely on thinning random samples and are thus not directly applicable to our setting.
\smallskip%

Our work also relates to recent advances in online discrepancy minimization \cite{AL-Saw, Bans-Spencer, Bans-JSS, Bans-JMSS}, which consider assigning a stream of \(2n\) points in \([0,1]^d\) to two sets \(X\) and \(Y\) (without discards) to minimize discrepancy \eqref{Dnn}. A variant of our algorithm may offer another approach to that 
problem -- see Proposition \ref{balancing_l1} below. Conversely, it would be interesting to see if their techniques could sharpen Theorem~\ref{main}. 

\begin{remark}[Open question]
It remains open whether the bound \(O(\log^{2d} n)\) in Theorem~\ref{main} is optimal. It would be interesting to either improve this bound or establish a matching lower bound.
\end{remark}

\begin{remark}[Dyadic boxes]
Our proof yields a stronger statement when one restricts to dyadic boxes. 
Fix a resolution parameter \(L \ge 1\). 
A dyadic interval is an interval of the form:
\[
\Bigl[ \frac{j}{2^\ell}, \frac{j+1}{2^\ell} \Bigr],
\qquad
j \in \{0,1,\ldots,2^\ell-1\},\ \ell \in \{0,1,\ldots,L-1\},
\]
and a dyadic box is a product of \(d\) dyadic intervals (possibly with different
\(j\) and \(\ell\) in different coordinates). In Step~1 of the proof of
Proposition~\ref{stream}, after thinning, we obtain the deterministic bound:
\[
\sup_{B\ \text{dyadic},\, 1 \le k \le m} \disc_k(B) \le T L^d.
\]
In particular, the two-sample discrepancy process restricted to dyadic boxes is
uniformly bounded by \(T L^d\) at 
all times and hence is subgaussian. 
This control is weakened in the final step of the proof,
where one passes from dyadic boxes to all axis-aligned boxes.
\end{remark}

\section{Reduction to the case of uniform marginals}\label{reduction}

It suffices to prove Proposition~\ref{stream} for distributions whose \(d\) coordinate marginals are uniform on \([0,1]\). We now explain why this entails no loss of generality.
\smallskip

Let \(F_1, \ldots, F_d\) be the marginal cumulative distribution functions. Without loss of generality, we may assume that each \(F_i\) is strictly increasing. This follows by perturbation: consider a mixture of the original distribution with a Gaussian on \(\mathbb{R}^d\). More precisely, flip a biased coin with success probability \(\theta\); if success, sample from the original distribution; otherwise, sample from the Gaussian. Samples of \(m\) points from the original and perturbed distributions coincide with probability \(\theta^m\), and so do their sign discrepancies. As \(\theta \to 1\), the expected discrepancy of the perturbed sample converges to that of the original. We may thus sample from the mixture instead of the original distribution, and the  marginals \(F_i\) of the mixture are strictly increasing.
\smallskip

For a real-valued random variable $X$ with cumulative distribution function \(F\), define the (randomized) integral transform of $X$ as a random convex combination of $F(X)$ and $F(X^-)$:
    $$
    \hat{X} 
    \coloneqq UF(X) + (1-U) F(X^-),
    $$
    where \(U\) is an independent random variable uniformly distributed on \([0,1]\), and $F(x^-) = \lim_{y \uparrow x} F(y)$. Then \(\hat{X}\) is uniformly distributed on \([0,1]\). Applying this transform coordinate-wise to any probability distribution on \(\mathbb{R}^d\) yields one with uniform marginals. 
\smallskip%

Let \(x_1,\ldots,x_m\) be an arbitrary set of points in \(\mathbb{R}^d\). For each \(x_i \in [0,1]^d\), consider its transformed version:
\[
\hat{x}_i(k) = U_i F_k \big(x_i(k)\big) + (1 - U_i) F_k \big(x_i(k)^-\big),
\]
where \(x_i(k)\) denotes the \(k\)th coordinate of \(x_i\), and similarly 
for \(\hat{x}_i\). We use its own independent copy \(U_i\) for each point. This preserves independence when the \(x_i\) are sampled independently from the original distribution. 
\begin{lemma}\label{u_transform}
If each \(F_k\) is strictly increasing, then the sign discrepancy of \(\hat{x}_1, \ldots, \hat{x}_m\) is no less than that of \(x_1, \ldots, x_m\).
\end{lemma}

\begin{proof}
To each axis-aligned half-infinite box
$B = (-\infty, b_1] \times \cdots \times (-\infty, b_d]$, associate the anchored box $\hat{B} = [0, F_1(b_1)] \times \cdots \times [0, F_d(b_d)]$. It suffices to show that for each \(i\),
\[
x_i \in B \iff \hat{x}_i \in \hat{B}.
\]
Indeed, if \(x_i \in B\), then \(\hat{x}_i \in \hat{B}\), regardless of \(U_i\). If \(x_i \notin B\), then \(x_i(k) > b_k\) 
for some \(k\). Hence, there is 
\(y\) with \(b_k < y < x_i(k)\). Then
$$
F_k(b_k) 
< F_k(y)
\le F_k(x_i(k)^{-})
\le \hat{x}_i(k),
$$
where the first two inequalities follow by the strict monotonicity of $F_k$, and the last bound follows from the definition of $\hat{x}_i(k)$ if one replaces $F_k(x_i(k))$ by the smaller value $F_k(x_i(k)^-)$ there. Thus, $\hat{x}_i \not\in \hat{B}$. The proof is complete. \qed
\end{proof}
\smallskip%

We conclude with the following statement.

\begin{lemma}
If Proposition~\ref{stream} holds for all Borel distributions with uniform marginals, then it holds for all Borel distributions.
\end{lemma}

\begin{proof}
Using the perturbation argument above, we reduce to the case where all marginals have strictly increasing distribution functions. Apply the randomized transform to obtain \(\hat{x}_1, \ldots, \hat{x}_m\) from \(x_1, \ldots, x_m\), using independent \(U_i\) to preserve independence.
\smallskip%

Let us discard some of the transformed points 
\(\hat{x}_i\) to reduce discrepancy. Then discard the corresponding indices from the original sample \(x_i\). By Lemma~\ref{u_transform}, which does not require independence, the discrepancy of the retained \(x_i\) is no greater than that of the retained \(\hat{x}_i\). This completes the proof. \qed
\end{proof}

\section{Proof of Proposition \ref{stream}}\label{proof}
Our proof relies on the following statement, which may be of independent interest.

\begin{proposition}\label{l1}
    Consider any vectors $v_1,\ldots,v_m \in \R^{N}$ and a number $T>0$. Let $\e_1,\ldots,\e_m$ be independent Rademacher random variables. There is an online algorithm to discard terms $\e_i v_i$ so that the running sums of the remaining terms $\e'_i v'_i$ satisfy 
    $$
    \sup_k \norm[\Big]{\sum_{i=1}^k \e'_i v'_i}_\infty
    \le T \quad \text{deterministically,}
    $$
    and
    $$
    \E\#\left\{
    \text{\normalfont{discarded terms}}
    \right\}
    \le \frac{1}{T} \sum_{i=1}^{m} \norm{v_i}_1.
    $$
\end{proposition}

The proof will need an elementary lemma.

\begin{lemma}   \label{lem: outside cube}
    
    Fix any vector $v \in \R^{N}$ and a number $T>0$. Let $w$ be a random vector uniformly distributed in the cube 
    $[-T/2,T/2]^{N}$. Then
    $$
    \Prob[\big]{\norm{w+v}_\infty > T/2}
    \le \frac{\norm{v}_1}{T}.
    $$
\end{lemma}

\begin{proof}
    By union bound, we have
    \begin{equation} \label{eq: infty norm union bound}
        \Prob[\big]{\norm{w+v}_\infty > T/2}
    \le \sum_{i=1}^{{N}} \Prob[\big]{\abs{w_i+v_i} > T/2}.
    \end{equation}
    For each $i$, since $w_i$ is uniformly distributed in $[-T/2,T/2]$ and $v_i$ is fixed, we have
    $$
    \Prob[\big]{\abs{w_i+v_i} > T/2}
    = \min\Big(\frac{\abs{v_i}}{T}, 1\Big)
    \le \frac{\abs{v_i}}{T}.
    $$
    Substitute this into \eqref{eq: infty norm union bound} to complete the proof. \qed
\end{proof}
\medskip%

\begin{proof}[Proof of Proposition \ref{l1}]
Let $K \coloneqq [-T/2,T/2]^{N}$. We arrange the discards as follows:
    \begin{itemize}
        \item Sample $v_0 \in K$ uniformly at random, and set $w_0 = v_0$.
\smallskip%

        \item Upon receiving $\e_k v_k$, check whether
        $$
        w_{k-1} + \e_k v_k 
        \in K.
        $$
        If so, accept the term $\e_k v_k$ and update $w_k \gets w_{k-1} + \varepsilon_k v_k$; otherwise, discard the term and set $w_k = w_{k-1}$.
\smallskip%

        \item Repeat for $\e_{k+1} v_{k+1}$.
    \end{itemize}
    This algorithm outputs a thinned subsequence of the original stream. By design, all $w_k$ remain in $K$ at all times $k=0,1,2,\ldots$, so all running sums of the accepted terms satisfy $\norm{w_k-w_0}_\infty \le \norm{w_k}_\infty + \norm{w_0}_\infty \le T/2 + T/2 \le T$, as claimed.
\smallskip%

    The random walk \(w_0, \ldots, w_{m}\) is a Markov chain on $K$. Due to the symmetry of the Rademacher distribution, the Markov chain is symmetric:
    \[
    \Prob*{w_{k} = w' \given w_{k-1} = w}
    =
    \Prob*{w_{k} = w \given w_{k-1} = w'}
    \quad \text{for each } w, w' \in K.
    \]
    Consequently, the uniform distribution on \(K\) is stationary. Since \(w_0\) is initialized uniformly, each \(w_k\) remains uniformly distributed on \(K\). By design, the term $\e_k v_k$ is discarded if and only if $\norm{w_{k-1} + \e_k v_k}_\infty > T/2$. 
    Thus:
    $$
    \E\#\left\{\text{discarded terms}\right\}
    = \sum_{k=1}^{m} \Prob[\big]{\norm{w_{k-1} + \e_k v_k}_\infty > T/2}
    \le \frac{1}{T} \sum_{k=1}^{m} \norm{v_k}_1,
    $$
    where the last bound follows from Lemma \ref{lem: outside cube}, since $w_{k-1}$ is uniformly distributed in $K$ and is independent of $\e_k$. The proof is complete. \qed
\end{proof}

\begin{proof}[Proof of Proposition \ref{stream}] 
We may assume that the  distribution is supported on $[0,1]^d$ and has uniform marginals -- see \S\,\ref{reduction}. The plan is to first bound the discrepancy over dyadic boxes using Proposition \ref{l1}, then extend the bound to lattice boxes via a union bound, and finally generalize it to all boxes through approximation. 
\smallskip%

Fix a resolution parameter $L \ge 1$; we will later choose it as $L = \log_2 n$ but it is simpler to keep it general for now. 

\medskip

{\em Step 1. Dyadic boxes.}
A dyadic interval is an interval of the type
$$
\left[ \frac{j}{2^\ell}, \frac{j+1}{2^\ell} \right]
\quad \text{where} \quad 
j \in \{0,1,\ldots,2^\ell-1\}
\text{ and } \ell \in \{0,1,\ldots,L-1\}.
$$
A dyadic box is the product of $d$ dyadic intervals, possibly with different $j$ and $\ell$ in different factors. To each point $x_i \in [0,1]^d$ of our sample, associate the vector that encodes which dyadic boxes the point belongs:
$$
v_i 
\coloneqq \left( \one_{\{x_i \in B\}} \right)_{B: \text{ dyadic box}} \in \R^{N},\quad N = (2^{L} - 1)^d.
$$
This vector is indexed by all dyadic boxes $B$, with each coordinate equal to $1$ if $x_i$ lies in $B$ and $0$ otherwise. It follows from the uniform marginal assumption that a.s. no \(x_i\) lies on the boundary of a dyadic box. Then: 
$$
\norm{v_i}_1
= \# \left\{ \text{dyadic boxes $B$ that contain $x_i$} \right\}
= L^d.
$$
Apply Proposition \ref{l1} with $TL^d$ instead of $T$. With 
$$
\E\#\left\{
\text{discarded terms}
\right\}
\le \frac{1}{TL^d} \sum_{i=1}^{m} \norm{v_i}_1
\le \frac{mL^d}{TL^d}
\le \frac{m}{T},
$$
as required, the remaining terms $\e'_i v'_i$ satisfy 
$$
\sup_k \norm[\Big]{\sum_{i=1}^k \e'_i v'_i}_\infty
\le TL^d \quad \text{deterministically}.
$$
By definition of $v_i$ above and sign discrepancy in \eqref{eq: sign discrepancy}, this means that the remaining points of the sample satisfy 
\begin{equation}    \label{eq: discrepancy dyadic}
    \disc_k(B) \le TL^d
    \quad \text{for each dyadic box $B$ and each $k$.}   
\end{equation}

\medskip

{\em Step 2. Lattice boxes.}
A lattice interval is an interval of the type
$$
\left[ 0, \frac{j+1}{2^{L-1}} \right]
\quad \text{where} \quad 
j \in \{0,1,\ldots,2^{L-1}-1\}.
$$
A lattice box is the product of $d$ lattice intervals, possibly with different $j$ in different factors.
\smallskip%

Any lattice interval can be partitioned into at most $L$ dyadic intervals. Thus, any lattice box can be partitioned into at most $L^d$ dyadic boxes. Therefore, summing $L^d$ bounds \eqref{eq: discrepancy dyadic} by triangle inequality, we conclude: 
\begin{equation}    \label{eq: discrepancy lattice}
    \disc_k(B) \le TL^{2d}
    \quad \text{for each lattice box $B$ and each $k$.}   
\end{equation}

\medskip

{\em Step 3. Slices.} We are about to approximate a general box by a dyadic box, with the approximation error controlled by the number of points that can fall into a thin slice. To prepare for this, let's bound that number.
\smallskip%

A slice is the product of one interval of the type  
\begin{equation}    \label{eq: smallest dyadic interval}
    \left[ \frac{j}{2^{L-1}}, \frac{j+1}{2^{L-1}} \right],
    \quad \text{where} \quad j \in \{0,1,\ldots,2^{L-1}-1\},   
\end{equation}
and $d-1$ full intervals $[0,1]$; the interval can appear as any of the $d$ factors.
\smallskip%

By the assumption made at the beginning of the proof, the distribution of the points $x_i$ has uniform marginals in $[0,1]$. Thus, for any slice $S$, the number of points falling in $S$, 
$$
Z_S \coloneqq \# \left\{ i \in [m]: x_i \in S \right\},
$$
has binomial distribution with parameters 
$m$ and $2^{-(L-1)}$. Since increasing \( m \) can only increase the number of points in slices, we assume the worst case \( m = 2n \). Applying Chernoff's inequality \cite[Theorem 2.3.1]{vershynin2018high}, we get for each $u > 0$,
\begin{equation}\label{chernoff}
\Prob*{Z_S > 16 n 2^{-(L-1)} + \ln(d2^{L-1}) + u}
\le e^{-\ln(d2^{L-1}) - u}
= \frac{e^{-u}}{d2^{L-1}}.
\end{equation}
See \S\,\ref{chernoff_section} for a direct 
proof. Taking the union bound over all 
$d2^{L-1}$ possible slices, we get 
$$
\Prob*{\max_S Z_S > 
16 n 2^{-(L-1)}+\ln(d2^{L-1}) + u}
\le e^{-u}.
$$
In other words, with probability at least $1-e^{-u}$, the following event holds: 
\begin{equation}\label{eq: slices}
    \max_{S: \text{ slice}} \# \left\{ i \in [2n]: x_i \in S \right\} 
    \le 16 n 2^{-(L-1)}+\ln(d2^{L-1})+u.
\end{equation}

{\em Step 4. General boxes.} Any interval $B = [0,b] \subset [0,1]$ can be approximated by some lattice interval $B' \supset B$ by rounding $b$ up to a lattice point; this way $B' \setminus B$ lies in some dyadic interval of the type \eqref{eq: smallest dyadic interval}. Therefore, a general box
$$
B = [0,b_1] \times \cdots \times [0,b_d]
$$
can be approximated by a lattice box $B' \supset B$ so that $B' \setminus B$ lies in a union of $d$ slices. By triangle inequality, it follows that
$$
\disc_k(B)
\le \disc_k(B') + d \cdot \max_{S: \text{ slice}} \# \left\{ i \in [2n]: x_i \in S \right\}.
$$
Taking supremum over axis-aligned boxes $B$ and $k$, and using \eqref{eq: discrepancy lattice}, we find:
\[
\sup_{B: \text{ box}, \; k} \disc_k(B)
\le TL^{2d} + d \cdot \max_{S: \text{ slice}} \# \left\{ i \in [2n]: x_i \in S \right\}.
\]
Set \(L = \log_{2} n\). From \eqref{eq: slices}, 
with probability \(1-e^{-u}\), we find:
\[
\sup_{B: \text{ box}, \; k} \disc_k(B) \leq 
T \log_{2}^{2d} n + d\ln(d n) + 
d(u + 32 - \ln{2}).
\]
Integrating the tail (see \cite[Lemma 1.2.1]{vershynin2018high}) completes the proof. \qed
\end{proof}

\section{On vector balancing}\label{balancing}
Let \(v_1, \ldots, v_m \in 
\mathbb R^{N}\). The (algorithmic) vector 
balancing problem asks for an algorithm 
to choose signs \(\varepsilon_1, \ldots, \varepsilon_m \in \{\pm 1\}\) 
so as to make the (sign) discrepancy
\[
\sup_k \left\| \sum_{i = 1}^{k} \varepsilon_i v_i \right\|_{\infty}
\]
small. It is perhaps interesting that Proposition 
\ref{l1} offers such an algorithm, albeit 
with a suboptimal bound. To begin with, 
let us recall the following simple bound.
\begin{lemma}\label{max-geometric-like}
Let $Y_1,\ldots,Y_m$ be nonnegative, 
integer-valued random variables and set
\[
\tau = \sup_{i} Y_i.
\]
Suppose that there exist numbers \(p_i \in [0,1/2]\) such that 
for all \(t \ge 1\) and \(i\),
\[
\Prob{Y_i > t} \leq p_i^t.
\]
Then
\[
\mathbb E \tau \leq 3 + \log_2(1 + S_1),\quad 
S_1 = \sum_{i} p_i.
\]
\end{lemma}
\begin{proof}
For each $t \geq 1$,
\[
\Prob{\tau > t}
= \Prob{\exists\, Y_i > t} \leq \sum_{i} \Prob{Y_i > t} \leq 
\sum_{i} p_i^{t}
\le  \sum_{i} 2^{-(t-1)} p_i
= 2^{-(t-1)} S_1.
\]
Then:
\[
\mathbb E \tau
= \sum_{t \geq 0} \Prob{\tau > t} \leq 1 + 
\sum_{t \geq 1} \min \left\{1, 2^{-(t-1)} S_1 \right\}.
\]
Set:
\[
t_0 = 1 + \log_2 S_1.
\]
If \(t_0 < 1\), then:
\[
\sum_{t \ge 1} \min\{1, 2^{-(t-1)} S_1\}
\le \sum_{t \ge 1} 2^{-(t-1)} S_1 \le 2 S_1 \le 2.
\]
If \(t_0 \geq 1\), then:
\[
\sum_{t \geq 1} 
\min\left\{ 1, 2^{-(t-1)} S_1 \right\} = 
\sum_{t \leq t_0} + \sum_{t > t_0} = t_0 + 
\sum_{t > t_0} 2^{-(t-1)} S_1 \leq 
t_0 + 2^{-(t_0 - 1)} S_1 = t_0 + 1 \leq 2 + \log_2 S_1.
\]
In both cases, 
\[
\mathbb E \tau \leq 3 + \log_2 (1 + S_1), 
\]
and the proof follows. \qed
\end{proof}
\smallskip%

\begin{proposition}\label{balancing_l1}
Consider any vectors $v_1,\ldots,v_m \in \R^{N}$. There is an online algorithm 
to choose \(\varepsilon_1, \ldots, \varepsilon_m = \pm 1\) so that the running sums satisfy
\begin{equation}\label{balancing_bound}
\mathbb E \sup_k \left\| \sum_{i = 1}^{k} \varepsilon_i v_i \right\|_{\infty} 
\leq 2 \sup_i \|v_i\|_1 \left[ 
3 + \log_2 \left(
1 + \frac{\sum_i \|v_i\|_1}{2 \sup_i \|v_i\|_1}
\right)
\right]
= O\bigl(\sup_i \|v_i\|_1 \log m\bigr).
\end{equation}
\end{proposition}
\smallskip%

The bound in Proposition~\ref{l1} is 
dimension free (independent of $N$) and extends to countable families $v_i$ with $\sum_i \|v_i\|_1 < \infty$. However, it is in terms of the $\ell^1$-norms of the vectors, which is weaker than the $\ell^2$-norm featured in~\cite{AL-Saw}.

\begin{proof}
To begin with, let us 
describe a non-streaming algorithm and then explain how to make it online. 
\smallskip%

Recall the discarding 
algorithm of Proposition \ref{l1}. 
It receives vectors \(v_1, \ldots, v_m\), samples a random \(v_0 \in K = [-T/2, T/2]^{N}\), and produces 
a thinned sequence \(v'_i, \varepsilon'_i\) with 
\[
\sup_k \left\| \sum_{i = 1}^{k} \varepsilon'_i v'_i \right\|_{\infty} \leq T.
\]
For each \(v_i\), the probability to be 
discarded (i.e., to receive \(\varepsilon_i = 0\)) 
is at most \(\|v_i\|_{1}/T\). 
\smallskip%

Set
\[
T = 2 \sup_i \| v_i \|_1,
\]
so that \(\|v_i\|_1/T \le 1/2\) for all \(i\).
\smallskip%

Let us now choose \(\varepsilon_i = \pm 1\) as 
follows:
\begin{itemize}
\item Sample \(v_0 \in K\) uniformly and apply the discarding algorithm to 
\(v_0, v_1, \ldots, v_m\). For each \(i\), 
if \(\varepsilon_i \neq 0\), keep this sign for 
\(v_i\). 
\smallskip%
\item If all \(\varepsilon_i \neq 0\), 
stop. Otherwise, let 
the vectors \(v_i\) with \(\varepsilon_i = 0\) 
form a new sequence and repeat the previous step, using fresh randomness.
\end{itemize}
\smallskip%

Let \(\tau\) be the number of iterations. Since 
\(\|v_i\|_1/T \le 1/2\) for all \(i\), 
\(\tau\) is almost surely finite. By triangle inequality,
\[
\sup_k \left\| \sum_{i = 1}^{k} 
\varepsilon_i v_i \right\|_{\infty} \leq \tau T.
\]
To bound \(\tau\), note that in each 
iteration, \(v_i\) is discarded with probability at most
\[
p_i = \frac{\|v_i\|_1}{T} \le \frac{1}{2}.
\]
We can write $\tau = \sup_i Y_i$ where $Y_i$ denotes the number of rounds in which $v_i$ is discarded. By independence, we have $\Prob{Y_i > t} \leq p_i^t$ for all $t \ge 1$, so Lemma~\ref{max-geometric-like} gives
\[
\mathbb E \tau \leq 3 + \log_2 (1 + S_1) 
\quad \text{where} \quad 
S_1 = \sum_i \frac{\|v_i\|_{1}}{T},
\]
and \eqref{balancing_bound} follows. 
\smallskip%

Finally, the construction can be 
implemented online. 
Instead of running the iterations 
sequentially, we run them in parallel. 
For each iteration \(\tau = 1,2,\ldots\), 
fix an independent instance 
of the discarding algorithm 
of Proposition~\ref{l1} with its own random seed 
and initial vector \(v_0\). When \(v_i\) arrives, we feed it 
first to round \(1\); if that round would discard it, we feed it to 
round \(2\), and so on, until some round accepts it and assigns a sign 
\(\varepsilon_i\). At that point we stop and never reconsider \(v_i\). 
Each decision for \(v_i\) depends only on \(v_1,\dots,v_{i}\), 
so the resulting algorithm is online. This completes the proof. \qed
\end{proof}

\section{Proof of \eqref{chernoff}}\label{chernoff_section}
Write $Z_S = Y_1 + \cdots + Y_m$, 
where $m \le 2n$ and $Y_i$ are i.i.d. with
\[
\mathbb P\{Y_i = 1\} = p = 2^{-(L-1)}, \quad \mathbb P\{Y_i = 0\} = 1-p.
\]
For each \(t > 0\) and \(a\),
\[
\mathbb P\{Z_S \ge a\}
= \mathbb P\{e^{t Z_S} \ge e^{t a}\}
\le e^{-t a} \,\E e^{t Z_S}
= e^{-t a} (1-p+pe^t)^m
\le \exp\big(-t a + m p (e^t - 1)\big).
\]
Set:
\[
a = 16 n p +\ln(d2^{L-1}) + u, \quad t = 1.
\]
Since \(m \leq 2n\), it follows that:
\[
-t a + m p (e^t - 1)
\le -a + 2np(e-1)
= 2np(e-9) - \ln(d2^{L-1}) - u
\le -\ln(d2^{L-1}) - u,
\]
Consequently,
\[
\mathbb P\{Z_S \ge 16 n p + \ln(d2^{L-1}) + u\}
\le \exp\big(-\ln(d2^{L-1}) - u\big)
= \frac{e^{-u}}{d2^{L-1}},
\]
and \eqref{chernoff} follows.

\bibliographystyle{plain}
\bibliography{ref}

\end{document}